\documentclass{amsart}
\usepackage{amssymb}
\usepackage{amsbsy}
\usepackage{amscd}
\usepackage{amsmath}
\usepackage{amsthm}
\usepackage{amsxtra}
\usepackage{latexsym}
\usepackage{tikz}
\usepackage{mathrsfs}
\usepackage{mathdots}
\usepackage{upgreek}
\usepackage{wasysym}
\usepackage[all,cmtip,dvipdfmx]{xy}
\usepackage{xcolor}
\definecolor{rouge}{rgb}{0.85,0.1,.4}
\definecolor{blue}{rgb}{0.1,0.2,0.9}
\definecolor{violet}{rgb}{0.7,0,0.8}

\usepackage{color}
\usepackage{colordvi}

\newcommand{\cprime}{$'$}

\usepackage[%dvipdfmx,
colorlinks=true,linkcolor=blue,urlcolor=violet,citecolor=rouge]{hyperref}

\newcommand{\on}{\operatorname}

\newcommand{\mc}{\mathcal}

\newcommand{\al}{\alpha}

\def\be{\beta}

\def\leq{\leqslant}
\def\geq{\geqslant}

\def\C{{\mathbb C}}
\def\Q{{\mathbb Q}}

\def\Z{{\mathbb Z}}

\def\1{{\bf 1}}

\def \End{{\rm End}}

\def \dep{{\rm depth}\,}

\def \<{\langle}
\def \>{\rangle}

\def \g{\frak g}
\def \fg{\widehat{\frak{g}}}

\def\gr{{\rm gr}V^{k}({\frak g})}

\def \h{\mathfrak{h}}

\def \W{\mathscr{W}}

\numberwithin{equation}{section}
\newtheorem{theorem}{Theorem}[section]
\newtheorem{prop}[theorem]{Proposition}
\newtheorem{lem}[theorem]{Lemma}

\theoremstyle{definition}
\newtheorem{defn}[theorem]{Definition}

\begin{document}
  \title{Simplicity of vacuum modules and associated varieties}
%\title[associated varieties of
%affine vertex algebras]{A remark on the associated variety of
%simple affine vertex algebras}

\author{Tomoyuki Arakawa}\thanks{Arakawa is supported by 
partially  supported 
by JSPS KAKENHI Grant Number
No.\ 17H01086 and No.\ 17K18724. }
\address[Arakawa]{ Research Institute for Mathematical Sciences, Kyoto University, Kyoto, 606-8502, Japan}
\email{arakawa@kurims.kyoto-u.ac.jp}

\author{Cuipo  Jiang} \thanks{Jiang is supported by CNSF grants 11771281 and  11531004}
\address[Jiang]{School of Mathematical Sciences,  Shanghai Jiao Tong University, Shanghai, 200240, China}
\email{cpjiang@sjtu.edu.cn}

\author{Anne Moreau}\thanks{Moreau is supported by the ANR Project GeoLie Grant number
ANR-15-CE40-0012, and by the Labex CEMPI (ANR-11-LABX-0007-01)}
\address[Moreau]{Univ. Lille, CNRS, UMR 8524 - Laboratoire Paul Painleve,
F-59000 Lille, France}
\email{anne.moreau@univ-lille.fr}

\maketitle

\begin{abstract}
In this note, we prove that the universal affine vertex algebra associated with
a  simple Lie algebra $\g$ is simple if and only if the associated
variety of its unique simple quotient is equal to $\g^*$.
We also derive an analogous result for the quantized
Drinfeld-Sokolov reduction applied to the universal affine vertex algebra.
 \end{abstract}

\section{Introduction}
Let $V$ be a vertex algebra,
and let
$$V\longrightarrow (\End V)[[z,z^{-1}]],\quad a\longmapsto a(z)=\sum_{n\in \Z}a_{(n)}z^{-n-1},
$$be the state-field correspondence.
The {\em Zhu $C_2$-algebra} \cite{Zhu96} of $V$ is by definition the quotient space
$R_V=V/C_2(V)$,
where $C_2(V)=\on{span}_{\C}\{a_{(-2)}b\mid a,b\in V\}$,
equipped with the Poisson algebra structure given by
$$\bar a. \bar b=\overline{a_{(-1)}b},\qquad \{\bar a,\bar b\}=\overline{a_{(0)}b},$$
for $a,b \in V$ with $\bar a := a+C_2(V)$.
The associated variety $X_V$ of $V$ is the reduced scheme
$X_V=\on{Specm}(R_V)$
corresponding to $R_V$.
It is a fundamental invariant of $V$ that captures %some
important properties
of the vertex algebra $V$ itself (see, for example, \cite{BeiFeiMaz,Zhu96,AbeBuhDon04,Miy04,Ara12,Ara09b,A2012Dec,AM15,AM16,Arakawa-Kawasetsu}).
Moreover,
it conjecturally  \cite{BeeRas} coincides with
the Higgs branch of
a 4D $\mc{N}=2$ supercoformal field theory $\mc{T}$
that is a hyperk\"{a}hler cone,
if $V$ corresponds to $\mc{T}$
by the 4D/2D duality discovered in \cite{BeeLemLie15}.

In the case that $V$ is
the universal affine vertex algebra $V^k(\g)$ at level $k \in\C$
associated with a complex finite-dimensional simple Lie algebra $\g$,
the variety $X_V$  is just the affine space $\g^*$ with Kirillov-Kostant Poisson structure.
In the case that $V$ is
the unique simple graded quotient $L_k(\g)$
of $V^k(\g)$,
the variety $X_V$ is a Poisson subscheme of $\g^*$
which is $G$-invariant and conic, where $G$ is the adjoint group of $\g$.

Note that if the level $k$ is irrational,
then
 $L_k(\g)=V^k(\g)$,
and hence $X_{L_k(\g)}=\g^*$.
More generally, if $L_k(\g)=V^k(\g)$, that is, $V^k(\g)$
is simple, then obviously
\hbox{$X_{L_k(\g)}=\g^*$.}

%Note that on the contrary to
%the associated variety of a primitive ideal of $U(\g)$,
%the variety $X_{L_k(\g)}$ is not necessarily
%contained in the nilpotent cone $\mathcal{N}(\g)$ of $\g$.
%In fact, $X_{L_k(\g)}=\g^*$ for a generic $k$ since
% in this case.

In this note, we prove that the converse is true.

\begin{theorem}\label{MainTheorem}
\label{Th:main}
The equality $L_k(\g)=V^k(\g)$ holds, that is,
$V^k(\g)$
is simple, if and only if $X_{L_k(\g)} = \g^*$.
\end{theorem}

It is known by  Gorelik and Kac \cite{GorKac07}
that $V^k(\g)$
is not simple if and only if
\begin{align}
r^{\vee}(k+h^{\vee}) \in \Q_{\geq 0}\backslash \left\{\frac{1}{m}\mid m\in \Z_{\geq 1}\right\},
\label{eq:Gorelik-Kac}
\end{align}
where
$h^{\vee}$ is the dual Coxeter number and
$r^{\vee}$ is the lacing number of $\g$.
Therefore,
Theorem \ref{MainTheorem} can be rephrased as
\begin{align}
X_{L_k(\g)}\subsetneq \g^*\iff \text{\eqref{eq:Gorelik-Kac} holds.}
\end{align}

Let us mention the cases when the variety $X_{L_k(\g)}$
is known
for $k$ satisfying \eqref{eq:Gorelik-Kac}.

First, it is known \cite{Zhu96,DonMas06} that $X_{L_k(\g)}=\{0\}$
 if and only if $L_k(\g)$ is integrable, that is,
 $k$ is a nonnegative integer.
Next, it is known that if $L_k(\g)$ is {\em admissible} \cite{KacWak89},
or equivalently,
if
\begin{align*}
k+h^{\vee}=\frac{p}{q},\quad p,q\in \Z_{\geq 1}, \ (p,q)=1,\
p\geq \begin{cases}h^{\vee}&\text{if }(r^{\vee},q)=1,\\
h&\text{if }(r^{\vee},q)\ne 1,
\end{cases}
\end{align*}
where $h$ is the Coxeter number of $\g$,
 then
 $X_{L_k(\g)}$ is the closure of some nilpotent orbit in $\g$ (\cite{Ara09b}).
 %It is also known that
 Further, it was observed in  \cite{AM15,AM17}
 that
there are   %In , the first and third authors found examples of
 cases when $L_k(\g)$ is
non-admissible %simple affine vertex algebras
and
$X_{L_k(\g)}$ is the closure of some nilpotent orbit.
% for
%certain non-admissible levels $k$.
 In fact,
 it was recently
conjectured in physics \cite{XieYan2019lisse} that,
 in view of the 4D/2D duality,
 there should be
 a large list of non-admissible simple affine vertex algebras
 whose associated varieties are the closures of some nilpotent orbits.
Finally,
there are also cases \cite{AM16} where $X_{L_k(\g)}$
 is neither $\g^*$ nor contained in the nilpotent cone $\mathcal{N}(\g)$ 
 of $\g$.

 In general, it is wide open to determine the variety $X_{L_k(\g)}$.

\smallskip

%Our proof of Theorem~\ref{Th:main} is purely combinatoric.
%Theorem ~\ref{Th:main} follows from

Now let us explain
the outline of
 the proof of Theorem~\ref{Th:main}.
 First, Theorem~\ref{Th:main} is known for the critical level $k=-h^\vee$
(\cite{FeiFre92,FreGai04}). Therefore Theorem~\ref{Th:main} follows from the following fact.
 \begin{theorem}\label{theorem:image-of-the-singular}
 Suppose that the level is non-critical, that is, $k\ne -h^\vee$.
The image of any nonzero singular vector  $v$  of $V^k(\g)$
in the Zhu $C_2$-algebra $R_{V^k(\g)}$ is nonzero.
 \end{theorem}

 The symbol $\sigma(w)$
of a  singular vector  $w$ in $V^k(\g)$ is
a %nontrivial
singular vector
 in the corresponding vertex Poisson algebra
${\rm gr}\, V^k(\g) \cong S(t^{-1}\g[t^{-1}])\cong \C[J_{\infty}\g^*]$,
where $J_\infty \g^*$ is the arc space of $\g^*$.
Theorem \ref{theorem:image-of-the-singular}
states that
the image of $\sigma(w)$
of a non-trivial singular vector $w$
under the natural projection
\begin{align}
\C[J_{\infty}\g^*]\longrightarrow \C[\g^*]=R_{V^k(\g)}
\label{eq:proj}
\end{align}
 is nonzero,
provided that $k$ is non-critical.
Hence,
Theorem \ref{theorem:image-of-the-singular}
would follow if
the image of any nontrivial singular vector
in $\C[J_{\infty}\g^*]$
under the projection \eqref{eq:proj}
is nonzero.
However, this is false (see Subsection \ref{subsection:remark}).
Therefore,
we do need to make use of the fact that
$\sigma(w)$ is the symbol of a  singular vector $w$ in $V^k(\g)$.
We also note that
the statement of Theorem \ref{theorem:image-of-the-singular}
is not true if $k$ is critical (see Subsection~\ref{subsection:remark}).

For this reason the proof of Theorem \ref{theorem:image-of-the-singular}
is  divided roughly into two parts.
First, we work in the commutative setting
to deduce
 a first important reduction (Lemma~\ref{lem:wi_-}).
Next, we use the Sugawara construction --
which is available only at non-critical levels -- %(see \S\ref{sec:Universal affine vertex algebras})
in the non-commutative setting
in order to complete the proof.
%to show that the image of
%a nontrivial singular vector in $V^k(\g)$ in $R_{V^k(\g)}$
%is nonzero.

\smallskip
Now, let us consider the {\em $W$-algebra} $\W^k(\g,f)$ associated with
a nilpotent element $f$ of $\g$ at the level $k$
defined by the generalized quantized Drinfeld-Sokolov reduction
\cite{FeiFre90,KacRoaWak03}:
$$\W^k(\g,f)=H^{0}_{DS,f}(V^k(\g)).$$
Here, $H^{\bullet}_{DS,f}(M)$ denotes the BRST
cohomology of the  generalized quantized Drinfeld-Sokolov reduction
associated with $f \in \mathcal{N}(\g)$ with coefficients in
a $V^k(\g)$-module~$M$.

By the Jacobson-Morosov theorem, $f$ embeds into an $\mathfrak{sl}_2$-triple
$(e,h,f)$. The  Slodowy slice $\mathscr{S}_{f}$ at $f$ is
 the affine
space $\mathscr{S}_{f}=f+\g^{e}$, where $\g^{e}$
is the centralizer of $e$  in $\g$.
It has a natural Poisson structure induced from that of $\g^*$ (see \cite{GanGin02}),
and we have \cite{DSK06,Ara09b} a natural isomorphism
$R_{\W^k(\g,f)}\cong \C[\mathscr{S}_{f}]$ of Poisson algebras, so that
\begin{align*}
 X_{\W^k(\g,f)}= \mathscr{S}_{f}.
\end{align*}
The natural surjection
$V^k(\g)\twoheadrightarrow L_k(\g)$
induces a surjection
$\W^k(\g,f)\twoheadrightarrow H_{DS,f}^0(L_k(\g))$
of vertex algebras (\cite{Ara09b}).
Hence the variety
$X_{H^{0}_{DS,f}(L_k(\g))}$
is a $\C^*$-invariant Poisson
subvarieties of the Slodowy slice $\mathscr{S}_f$.

Conjecturally  \cite{KacRoaWak03,KacWak08},
the vertex algebra
$H^{0}_{DS,f}(L_k(\g))$ coincides
the unique simple (graded) quotient $\W_k(\g,f)$
of  $\W^k(\g,f)$
provided that
\hbox{$H^{0}_{DS,f}(L_k(\g))\ne 0$.}
(This conjecture has been verified in many cases \cite{Ara05,Ara07,Ara08-a,AEkeren19}.)

As a consequence of Theorem \ref{Th:main}, we obtain the following
result.

\begin{theorem}\label{Th:main2}
Let $f$ be any nilpotent element of $\g$.
The following assertions are equivalent:
\begin{enumerate}
\item $V^k(\g)$ is simple,
\item $\W^k(\g,f)=H^{0}_{DS,f}(L_k(\g))$,
\item $X_{H^{0}_{DS,f}(L_k(\g))} = \mathscr{S}_f$.
\end{enumerate}
%
%
%The equality $H^{0}_{DS,f}(L_k(\g))=H^{0}_{DS,f}(V^k(\g))$ holds
%if and only if $X_{H^{0}_{DS,f}(L_k(\g))} = \mathscr{S}_f$.
%\textcolor{red}{I removed the last statement because it does not follow from the above (Need to consider the case when $H^{0}_{DS,f}(L_k(\g))=0$) .}
%In particular, $\W^k(\g,f)$ is simple if
%and only if $X_{H^{0}_{DS,f}(L_k(\g))} = \mathscr{S}_f$,
%up to the above conjecture.
\end{theorem}
%If $X_{H^{0}_{DS,f}(L_k(\g))} = \mathscr{S}_f$ then
%$H^{0}_{DS,f}(L_k(\g)) \not=0$ (see Proposition \ref{}),
%and so, $\W_k(\g,f) \cong \W^k(\g,f)$ is simple by Theorem \ref{Th:main2},
%up the above conjecture. The converse is obviously true.
Note that Theorem \ref{Th:main2} implies that
 $V^k(\g)$  is simple
if
$X_{\W_k(\g,f)}=\mathscr{S}_f$
and $H^{0}_{DS,f}(L_k(\g))\ne 0$
since
$  X_{H_{DS,f}^0(L_k(\g))}\supset X_{\W^k(\g,f)}$.

\smallskip

The remainder of the paper is structured as follows.
%Section \ref{sec:preliminaries}  gives a condensed exposition of
%vertex algebras and corresponding graded vertex Poisson algebras.
In Section \ref{sec:affine_vertex_algebras} we set up notation
in the case of affine vertex algebras that  will be the framework of this note.
Section \ref{sec:main_proof} is devoted to the proof of Theorem~\ref{Th:main}.
%Our proof is purely combinatoric. The basic idea is to show that, given a nontrivial
%singular vector
%of $V^k(\g)$, for $k \not=-h^\vee$, where $h^\vee$ is the dual Coxeter number,
%its image in the Zhu $C_2$-algebra of $V_k(\g)$ by the canonical
%projection from  $V^k(\g)$ onto $R_{V^k(\g)}$ is nonzero.
%We observe that a nontrivial singular vector in $V^k(\g)$ induces
%a nontrivial singular vector in the corresponding vertex Poisson algebra
%${\rm gr} V^k(\g) \cong S(t^{-1}\g[t^{-1}])$ and that
%$R_{V^k(\g)}=R_{{\rm gr}\,V^k(\g) }$. Hence our arguments
%are facilitated by working in the commutative setting.
%(Note that Theorem~\ref{Th:main} is known for the critical level $k=-h^\vee$.)
In Section \ref{sec:W-algebras}, we have compiled some known facts
on Slodowy slices, $W$-algebras and their associated varieties.
Theorem~\ref{Th:main2} is proven in this section.
%In Section~\ref{sec:open_problems}, we discuss other related problems.

\subsection*{Acknowledgements} 
T.A.~and A.M.~like to thank warmly Shanghai Jiao Tong University 
for its hospitality during their stay in September, 2019.

\section{Universal affine vertex algebras and associated
graded vertex Poisson algebras}
\label{sec:affine_vertex_algebras}
Let
$\fg$ be the affine Kac-Moody algebra associated with $\g$, that is,
\begin{align*}
\fg =\g [t,t^{-1}]\oplus  \C K,
\end{align*}
where the commutation relations are given by
\begin{align*}
[x\otimes t^m,y\otimes t^n]=[x,y]\otimes t^{m+n}+m(x|y)\delta_{m+n,0}K,\quad
[K,\fg]=0,
\end{align*}
for $x,y\in \g$ and $m,n\in \Z$.
Here,
$$(~|~)=\displaystyle{\frac{1}{2h^\vee}\times} \text{ Killing form of }\g$$
is the usual normalized inner product.
For $x \in \g$ and $m \in \Z$, we shall write $x(m)$ for $x
\otimes t^m$.
%There is an isomorphism $\nu$ from $\frak{h}$ to $\frak{h}^*$ via $(~|~)$. We will identify $\nu^{-1}(\lambda)$ with $\lambda$, for $\lambda\in\frak{h}^*$.

%, and with constant structures $c'_{\al,\beta}$.

%The following easy lemma, which is an immediate consequence of Exercise 9.10 and Lemma 25.2 in \cite{Hu}, will be used later on.
%\begin{lem}
%\label{l3.01}
%For $\al,\be\in{\Delta}_{+}$ such that $\alpha+\beta\in{\Delta}$,  and $m\geq 1$, we have
%\begin{align*}
%& [e_{\al}(-1)f_{\al}(0)+e_{\be}(-1)f_{\be}(0), e_{\al+\be}(-m)]   \\
%& \hspace{2cm} =  \; [e_{\al}(-1), e_{\al+\be}(-m)]f_{\al}(0)+[e_{\be}(-1), e_{\al+\be}(-m)]f_{\be}(0)  \\
%& \hspace{2.5cm} + \; c'_{-\al,\al+\be}(e_{\al}(-1)e_{\be}(-m)-e_{\be}(-1)e_{\al}(-m)),
%\end{align*}
%for some $c_{-\al,\al+\be}\in \R^*=\R\backslash\{0\}$. In particular,
%\begin{align}
% \label{eqe4}
%[e_{\al}(-1)f_{\al}(0)+e_{\be}(-1)f_{\be}(0), e_{\al+\be}(-1)]
%= & \; [e_{\al}(-1), e_{\al+\be}(-1)]f_{\al}(0)+[e_{\be}(-1), e_{\al+\be}(-1)]f_{\be}(0) &  \\\nonumber
% & + \; c'_{-\al,\al+\be}c'_{\al,\be}e_{\al+\be}(-2), &
%\end{align}
%for some $c'_{\al,\be}, c'_{-\al,\al+\be}\in \R^*$.
%\end{lem}

\subsection{Universal affine vertex algebras}
\label{sec:Universal affine vertex algebras}
For $k \in \C$, set
\begin{align*}
V^k(\g)=U(\fg)\otimes _{U(\g [t]\oplus  \C K)}\C_k,
\end{align*}
where $\C_k$ is the one-dimensional representation of $\g [t]\oplus  \C K$
on which $K$ acts as multiplication by $k$ and $\g\otimes  \C[t]$ acts trivially.

By the Poincar\'{e}-Birkhoff-Witt Theorem, the direct sum decomposition,
we have
\begin{align}
V^k(\g)\cong U(\g \otimes t^{-1}\C[t^{-1}]) = U(t^{-1} \g[t^{-1}]).
\label{eq:PBW}
\end{align}

The space $V^k(\g)$ is naturally graded,
\begin{align*}
V^k(\g) =\bigoplus_{\Delta\in \Z_{\geq 0}}V^k(\g) _{\Delta},
\end{align*}
where the grading is defined by
$$\deg (x^{i_1}(-n_1)\ldots x^{i_r}(-n_r) {\bf 1}) =  \sum_{i=1}^r n_i,
\quad r \geq 0, \; x^{i_j} \in \g,
$$
with ${\bf 1}$ the image of $1\otimes 1$ in $V^k(\g)$.
We have $V^k(\g)_0=\C{\bf 1} $,
and we identify $\g$ with $V^k(\g)_1$ via the linear isomorphism
defined by $x\mapsto x(-1){\bf 1} $.

It is well-known that $V^k(\g)$ has a unique vertex algebra structure
such that ${\bf 1}$ is the vacuum vector,
$$x(z) := Y(x\otimes t^{-1},z)  =\sum\limits_{n \in \Z} x(n) z^{-n-1},$$
and
\begin{align*}
[T,x(z)]=\partial_z x(z)
\end{align*}
for $x \in \g$,
where $T$ is the translation operator.
Here, $x(n)$ acts on $V^k(\g)$ by left multiplication, and
so, one can view $x(n)$ as an endomorphism of $V^k(\g)$.
The vertex algebra
$V^k(\g)$ is called the {\em universal affine vertex algebra}
associated with $\g$ at level~$k$ \cite{FZ,Zhu96,LL}.

The vertex algebra
$V^k(\g)$ is a vertex operator algebra, provided that $k+h^\vee\not=0$,
by the {\em Sugawara construction}.
More specifically, set
$$S=\displaystyle{\frac{1}{2}} \sum_{i=1}^{d}
x_{i}(-1) x^{i}(-1)  {\bf 1},$$
where $\{x_{i}\colon i=1,\ldots,d\}$ is the dual
basis of a basis $\{x^{i}\colon i=1,\ldots,\dim \g\}$ of $\g$
with respect to the bilinear form $(~|~)$, with $d = \dim \g$.
%and $$x^{i}(z)=\sum_{n\in\Z}x^{i}_{(n)} z^{-n-1},
%\qquad
%x_{i}(z)=\sum_{n\in\Z} x_{i,(n)} z^{-n-1}.$$
Then for $k \not=-h^\vee$, the vector
$\omega=\displaystyle{\frac{S}{k+h^\vee}}$
is a conformal vector of $V^k(\g)$ %, called the {\em Segal-Sugawaravector},
with central charge
$$c(k)=\displaystyle{\frac{k \dim \g}{k+h^\vee}}.$$
Note that, writing $\omega(z) = \sum\limits_{n \in \Z} L_n z^{-n-2}$,
we have
\begin{align*}
L_0= \dfrac{1}{2(k+h^\vee)}
\left( \sum\limits_{i=1}^{d}x_i(0)x^i(0)+\sum\limits_{n=1}^{\infty}\sum\limits_{i=1}^{d}(x_i(-n)x^i(n)+x^i(-n)x_i(n)) \right),
\end{align*}
\begin{align*}
L_n=\dfrac{1}{2(k+h^\vee)}
\left(\sum\limits_{m=1}^{\infty}\sum\limits_{i=1}^{d}x_i(-m)x^i(m+n)
+\sum\limits_{m=0}^{\infty}\sum\limits_{i=1}^{d}x^i(-m+n)x_i(m)\right), \quad
\text{if} \ n\neq 0.
\end{align*}

\begin{lem}[{\cite{Kac_infinite}}]
\label{lem:sugawara-vs-g}
We have
$$[L_n,x(m)]= - m x(m+n), \quad
\text{ for }x\in \g,\ m,n\in \Z,$$
and $L_n\mathbf{1}=0$ for $n\geq -1$.
\end{lem}
We have
$V^k(\g) _{\Delta}=\{v\in V^k(\g)\mid L_0 v=\Delta v\}$
and $T=L_{-1}$ on $V^k(\g)$, provided that $k+h^{\vee}\ne 0$.

{Any
graded quotient of $V^k(\g)$
as $\fg$-module
has the structure of a quotient vertex algebra.
In particular,
the  unique simple graded quotient $L_k(\g)$
is a vertex algebra,
and is called the {\em simple affine vertex algebra associated with $\g$ at level $k$. }}

\subsection{Associate graded vertex Poisson algebras of affine vertex algebras}
It is known by Li \cite{Li05}
that any vertex algebra $V$ admits a canonical filtration $F^\bullet V$,
called the {\em Li filtration} of $V$.
For a quotient $V$ of $V^k(\g)$,
$F^\bullet V$ is described as follows.
The subspace
 $F^p V$ is spanned by the elements
$$
y_{1}(-n_1-1)\cdots y_{r}(-n_r-1)\mathbf{1}
$$
with $y_{i} \in \g$,
$n_i\in\Z_{\geq 0}$, $ n_1+\cdots +n_r\geq p$.
We have
\begin{align}\nonumber
& V=F^0V\supset F^1V\supset\cdots, \quad \bigcap_{p}F^pV=0,\\ \label{eq:translation}
& TF^pV\subset F^{p+1}V,\\\nonumber
& a_{(n)}F^{q}V\subset F^{p+q-n-1}V \ for \ a\in F^{p}V, \ n\in\Z,\\\nonumber
& a_{(n)}F^{q}V\subset F^{p+q-n}V \ for \ a\in F^{p}V, \ n\geq 0.
\end{align}
Here we have set $F^pV=V$ for $p<0$.

Let ${\rm gr}^FV=\bigoplus_p F^pV/F^{p+1}V$ be the associated graded vector space.
The space ${\rm gr}^FV$ is a vertex Poisson algebra by
\begin{align*}
& \sigma_{p}(a)\sigma_{q}(b)=\sigma_{p+q}(a_{(-1)}b),\\
& T\sigma_{p}(a)=\sigma_{p+1}(Ta),\\
& \sigma_{p}(a)_{(n)}\sigma_{q}(b)=\sigma_{p+q- n}(a_{(n)}b)
\end{align*}
for $a,b\in V$,
$n\geq 0$,
where $\sigma_p \colon F^p(V)\rightarrow F^pV/F^{p+1}V$ is the principal symbol map.
In particular,
$\on{gr}^F V$ is a $\g[t]$-module by the correspondence
\begin{align}
\g[t]\ni x(n)\longmapsto \sigma_0(x)_{(n)}\in \End(\on{gr}^F V)
\label{eq:g[t]-action-in-gr}
\end{align}
for $x\in \g$, $n\geq 0$.

The filtration  $F^\bullet V$
is compatible with the grading:
$F^pV=\bigoplus\limits_{\Delta\in \Z_{\geq 0}}F^pV_{\Delta}$,
where
 $F^p V_\Delta := V_\Delta \cap F^p V$.

Let $U_{\bullet}(t^{-1}\g[t^{-1}])$ be the PBW filtration of
$U(t^{-1}\g[t^{-1}])$,
that is,
$U_{p}(t^{-1}\g[t^{-1}])$
is the subspace of $U(t^{-1}\g[t^{-1}])$
spanned by monomials $y_{1} y_{2} \dots y_{r}$ with $y_i \in \g$,
$r\leq p$.
Define
\begin{align*}
G_p V= U_{p}(t^{-1}\g[t^{-1}])\mathbf{1}.
\end{align*}
Then $G_\bullet V$ defines an increasing filtration of $V$.
We have
\begin{align}
F^p V_{\Delta}=G_{\Delta-p}G_{\Delta},
\end{align}
where
 $G_p V_{\Delta}:=G_p V\cap V_{\Delta}$,
see \cite[Proposition 2.6.1]{Ara12}.
Therefore,
the graded space ${\rm gr}^G V = \bigoplus\limits_{p \in \Z_{\geq 0}}
G_p V/G_{p-1}V$ is isomorphic to $\on{gr}^F V$.
In particular,
we have
\begin{align*}
\gr\cong \on{gr}U_{\bullet}(t^{-1}\g[t^{-1}])\cong S(t^{-1}\g[t^{-1}]).
\end{align*}
The action of $\g[t]$ on $\gr=S(t^{-1}\g[t^{-1}])$
coincides with the one
induced from the action of $\g[t]$ on
$\g[t,t^{-1}]/\g[t] \cong t^{-1}\g[t^{-1}]$.
More precisely,
the element $x(m)$, for $x \in \g$ and $m \in \Z_{\geq 0}$,
acts on $S(t^{-1}\g[t^{-1}])$ as follows:
\begin{align}\nonumber
& x(m).{\bf 1} =0, & \\\label{eq:action}
& x(m).v  = \sum\limits_{j=1}^r
 \sum\limits_{n_{j}-m >0}  y_{1}(-n_1)\ldots %x^{i_{j-1}}(-n_{j-1})
[x,y_{{j}}] (m- n_{j})
%x^{i_{j+1}}(-n_{j+1})
\ldots y_{r}(-n_r), &
\end{align}
if $v =y_{1}(-n_1)\ldots y_{r}(-n_r)$
with $y_i \in \g$, $n_1,\ldots,n_r \in \Z_{>0}$.

\subsection{Zhu's $C_2$-algebras and associated varieties
of affine vertex algebras}
\label{sub:Zhu affine vertex algebras}

We have \cite[Lemma 2.9]{Li05}
$$F^p V = {\rm span}_\C\{ a_{(-i-1)} b \colon a \in V, i \geq 1, b \in F^{p-i} V \}$$
for all $p \geq 1$. In particular,
$$F^1 V = C_2(V),$$
where $C_2(V)=\on{span}_{\C}\{a_{(-2)}b\mid a,b\in V\}$.
Set
$$R_V = V/C_2(V) = F^0 V / F^1 V \subset {\rm gr}^FV.$$
It is known by Zhu \cite{Zhu96} that $R_V$ is a Poisson algebra.
The Poisson algebra structure can be understood as the restriction of the vertex Poisson structure of ${\rm gr}^FV$. It is given by
$$\bar a \cdot \bar b = \overline{a_{(-1)} b}, \quad \{\bar a, \bar b\}= \overline{a_{(0)}b},$$
for $a, b \in V$, where $\bar a = a+C_2(V)$.

By definition \cite{Ara12}, the
{\em associated variety} of $V$
is
the reduced scheme $$X_V:= {\rm Specm}(R_V).$$

It is easily seen that
$$F^1 V^k(\g) = C_2(V^k(\g)) = t^{-2} \g[t^{-1}] V^k(\g).$$

The following map defines an isomorphism of Poisson algebras
\begin{align*}
\begin{array}[t]{rcl}
\C[\g^*] \cong S(\g) & \longrightarrow & R_{V^k(\g)}  \\[0.2em]
\g \ni x  & \longmapsto  & x (-1) {\bf 1}
+ t^{-2} \g[t^{-1}] V^k(\g). 
\end{array}
\end{align*}
Therefore, $R_{V^k(\g)}  \cong \C[\g^*]$ and so, $
X_{V^k(\g)}\cong \g^*$.

More generally, if $V$ is a quotient of $V^k(\g)$
by some ideal $N$, then
we have
\begin{align}
R_{V}\cong \C[\g^*]/I_N
\end{align}
as Poisson algebras,
where $I_N$ is the image of $N$ in $R_{V^k(\g)}=\C[\g^*]$.
%$V/C_2(V)=V/\g[t^{-1}]t^{-2} V$
%and we have a surjective Poisson algebra homomorphism
%\begin{align*}
%\begin{array}[t]{rcl}
% \C[\g^*]=S(\g) & \longtwoheadrightarrow & R_V=V/t^{-2}\g[t^{-1}]V \\[0.2em]
%x & \longmapsto  & \overline{x (-1){\bf 1}} + t^{-2}\g[t^{-1}]V,
%\end{array}
%\end{align*}
%where $\overline{x (-1){\bf 1}} $ denotes the image of $x(-1){\bf 1}$ in the quotient of
%$V$.
Then $X_V$ is just the zero locus of $I_N$ in
$\g^*$.
It is a closed $G$-invariant conic subset of $\g^*$.

Identifying $\g^*$ with $\g$ through the bilinear form $(~|~)$,
one may view $X_V$ as a subvariety of $\g$.

\subsection{PBW basis}
Let ${\Delta}_{+}=\{\be_1,\cdots,\be_q\}$ be the set of positive roots for $\g$ with respect
to a triangular decomposition
$\g= \mathfrak{n}_- \oplus \mathfrak{h} \oplus \mathfrak{n}_+$,
where $q=(d-\ell)/2$ and $\ell={\rm rk}(\g)$.

Form now on, we
fix a basis
$$\{u^i,  e_{\be_j}, f_{\be_j}\colon\  i=1,\ldots,\ell, \, j=1,\ldots, q\}$$
of $\g$ such that $\{u^{i} \colon i=1,\ldots,\ell\}$
is an orthonormal basis of $\h$ with respect to $(~|~)$ and
$(e_{\be_i}|f_{\be_i})=1$ for $i=1,2,\cdots,q$.
In particular, $[e_{\be_i}, f_{\be_i}]= \be_i$
for $i=1,\ldots,q$ (see, for example, \cite[Proposition 8.3]{Humphreys}),
where $\h^*$ and $\h$ are identified through $(~|~)$.
One may also assume that ${\rm ht}(\be_i)\leq {\rm ht}(\be_j)$ for $i< j$,
%\commentT{Are you sure of this order?}
where ${\rm ht}(\be_i)$ stands for the height of the positive root $\be_i$.

We define the structure constants $c_{\al,\be}$ by
$$[e_\al,e_{\be}] = c_{\al,\be} e_{\al+\be},$$
provided that $\al$, $\be$ and $\al+\be$ are in $\Delta$.
Our convention is that $e_{-\al}$ stands for $f_\al$ if $\alpha  \in \Delta_+$.
If $\al$, $\be$ and $\al+\be$ are in $\Delta_+$,
then from the equalities,
\begin{align*}
c_{-\al,\al+\be} = ( f_{\be} | [f_{\al},e_{\al+\be}] ) = - ( f_{\be} | [e_{\al+\be},f_{\al}] )
= - ( [f_{\be}, e_{\al+\be}]|f_{\al}) = -  c_{-\be,\al+\be},
\end{align*}
we get that
\begin{align}
\label{eq:structure_constant}
c_{-\al,\al+\be}= -  c_{-\be,\al+\be}.
\end{align}

By \eqref{eq:PBW},
%By the PBW theorem, $V^k(\g) \cong U(t^{-1} \g[t^{-1}])$
% as $\C$-vector spaces. Therefore
the above basis of $\g$
 induces a basis of $V^k(\g)$ consisted of ${\bf 1}$
 and the elements of the form
  \begin{align}
 \label{eq:PBW_basis}
 z = z^{(+)}  z^{(-)}  z^{(0)}  {\bf 1},
 \end{align}
with
 \begin{align*}
 & z^{(+)} := e_{\be_{1}}(-1)^{a_{1,1}}\cdots
e_{\be_{1}} (- r_1)^{a_{1,r_1}}\cdots
e_{\be_q}(- 1)^{a_{q,1}}\cdots e_{\be_q}(- r_q)^{a_{q,r_q}} , \\ \nonumber
& z^{(-)} : =f_{\be_1}(-1)^{b_{1,1}}\cdots f_{\be_1}(- s_1)^{b_{1,s_1}}\cdots
f_{\be_q}(- 1)^{b_{q,1}}
\cdots f_{\be_q}(- s_q)^{b_{q,s_q}} , \\\nonumber
&  z^{(0)} := u^1(- 1)^{c_{1,1}}\cdots u^1(- t_1)^{c_{1,t_1}}
\cdots u^\ell (- 1)^{c_{\ell,1}}\cdots u^\ell(- t_\ell )^{c_{\ell, t_\ell}},
\end{align*}
where $r_1,\ldots,r_q, s_1,\ldots,s_q,,t_1,\ldots,t_\ell$
are positive integers, and $a_{l,m},b_{l,n},c_{i,j}$, for $l =1,\ldots,q$,
$m=1,\ldots,r_l$, $n=1,\ldots,s_l$,  $i = 1,\ldots,\ell$, $j=1,\ldots,t_i$
are nonnegative integers such
that at least one of them is nonzero.

\begin{defn}
\label{def:monomial_V}
Each element $x$ of $V^k(\g)$
is a linear combination of elements in the above PBW basis,
each of them will be called a {\em PBW monomial} of $x$.
\end{defn}

\begin{defn}
\label{def:depth_V}
For a PBW monomial $v$ as in \eqref{eq:PBW_basis},
we call {\em depth} of $v$ the integer
\begin{align*}
\dep(v) &=\sum\limits_{i=1}^q \left( \sum\limits_{j=1}^{r_i}
a_{i,j}(j-1) + \sum\limits_{j=1}^{s_i} b_{i,j} (j-1) \right)
+  \sum\limits_{i=1}^\ell \sum\limits_{j=1}^{t_i} c_{i,j}(j-1).&
\end{align*}
In other words, a PBW monomial $v$ has depth $p$ means that
$v \in F^{p}V^k(\g)$ and $v \not \in F^{p+1}V^k(\g)$.

By convention, $\dep({\bf 1})=0$.

%For $x \in V^k(\g)$, we shall say that $x$ of {\em depth $p \in\Z_{\geq 0}$}
%if $x$ a linear combination of linearly independent monomials of depth $p$.

For a PBW monomial $v$ as in \eqref{eq:PBW_basis},
we call {\em degree} of $v$ the integer
\begin{align*}
\deg(v) &=
 \sum\limits_{i=1}^q \left( \sum\limits_{j=1}^{r_i}
a_{i,j} + \sum\limits_{j=1}^{s_i} b_{i,j}\right)
+ \sum\limits_{i=1}^\ell \sum\limits_{j=1}^{t_i} c_{i,j} ,&
\end{align*}
In other words, $v$ has degree $p$ means that
$v \in G_{p}V^k(\g)$ and $v \not \in G_{p-1}V^k(\g)$
since the PBW filtration of $V^k(\g)$ coincides with the
standard filtration $G_\bullet V^k(\g)$.

By convention, $\deg({\bf 1})=0$.
\end{defn}

Recall that a {\em singular vector} of a $\g[t]$-representation $M$
is a vector $m \in M$
such that $e_{\alpha}(0).m=0$, for all $\alpha \in {\Delta}_{+}$,
and $f_\theta(1).m=0$, where $\theta$ is the highest positive root of $\g$.

From the identity
\begin{align*}
L_{-1}=
&  \dfrac{1}{k+h^\vee}  \left( \sum\limits_{i=1}^{\ell}\sum\limits_{m=0}^{\infty} u^i(-1-m) u^i(m)+\sum\limits_{\al\in{\Delta}_{+}}\sum\limits_{m=0}^{\infty}
(e_{\al}(-1-m)f_{\al}(m)+f_{\al}(-1-m)e_{\al}(m))\right),
\end{align*}
we deduce the following easy observation, which will be useful
in the proof of the main result.

\begin{lem}
\label{lem:Sugawara_singular_vector}
If $w$ is a singular vector of $V^k(\g)$,
then
$$
L_{-1} w=\dfrac{1}{k+h^\vee} \left(\sum\limits_{i=1}^{\ell}u^i(-1)u^i(0)
+\sum\limits_{\al\in{\Delta}_{+}}e_{\al}(-1)f_{\al}(0)\right) w \, .$$
\end{lem}
\subsection{Basis of associated graded vertex Poisson algebras}
%Since the PBW filtration of $V^k(\g)$ coincides with the
%standard filtration $G_\bullet V^k(\g)$,
%by \eqref{eq:standard_Li-coincide}, we get that
%\begin{align*}
%S(t^{-1}\g[t^{-1}]) \cong {\rm gr}^G V^k(\g)
%\cong {\rm gr}^F V^k(\g),
%\end{align*}
%as vertex Poisson algebras.

Note that $\gr=S(t^{-1}\g[t^{-1}])$ has a basis consisting
of ${\bf 1}$ and elements of the form \eqref{eq:PBW_basis}.
Similarly to Definition \ref{def:monomial_V}, we have the following
definition.

\begin{defn}
\label{def:monomial_S}
Each element $x$ of $S(t^{-1}\g[t^{-1}])$
is a linear combination of elements in the above basis,
each of them will be called a {\em monomial} of $x$.
\end{defn}

%
%Since $V^k(\g)$ is a vertex module over itself,
%$S(t^{-1}\g[t^{-1}]) \cong {\rm gr}^F V^k(\g)$
%is a vertex Poisson module over itself.
%In particular, this gives to $S(t^{-1}\g[t^{-1}])$ a
%$\g[t]$-module structure.

%We observe that if $w$ is a singular vector of $V^k(\g)$,
%then $\sigma(w)$ is a singular vector of $S(t^{-1}\g[t^{-1}])$,
%and the conclusion of Lemma \ref{lem:}

{As in the case of $V^k(\g)$,}
the space $S(t^{-1}\g[t^{-1}])$ has two natural gradations.
The first one is induced from the degree of elements as
polynomials.
% over $J_{\infty}\g^*=\g^*[[t]]$. Note indeed
%that  $S(t^{-1}\g[t^{-1}]) \cong
%\C[\g[[t]]]$.
We shall write $\deg(v)$ for the degree of a homogeneous
element $v \in S(t^{-1}\g[t^{-1}])$ with respect to this gradation.

The second one is induced from the Li filtration via the isomorphism
$S(t^{-1}\g[t^{-1}]) \cong {\rm gr}^F V^k(\g)$.
%The following terminology is coherent with Definition \ref{def:depth_V}.
%
%\begin{defn}
%\label{def:depth}
The degree of a homogeneous
element $v \in S(t^{-1}\g[t^{-1}])$ with respect to the gradation
induced by Li filtration
will be called the {\em depth} of $v$, and will be denoted by $\dep(v)$.
%\end{defn}

Notice that any element $v$ of the form \eqref{eq:PBW_basis} is homogenous
for both gradations.
%\begin{align*}
%\deg(v) &= \sum\limits_{i=1}^\ell \sum\limits_{j=1}^{r_i} a_{i,j}+
% \sum\limits_{i=1}^q \left( \sum\limits_{j=1}^{s_i}
%b_{i,j} + \sum\limits_{j=1}^{t_i} c_{i,j}\right),&\\
%\dep(v) &= \sum\limits_{i=1}^\ell \sum\limits_{j=1}^{r_i} a_{i,j}(j-1)
%+ \sum\limits_{i=1}^q \left( \sum\limits_{j=1}^{s_i}
%b_{i,j}(j-1) + \sum\limits_{j=1}^{t_i} c_{i,j} (j-1) \right).&
%\end{align*}
By convention, $\deg({\bf 1})=\dep({\bf 1})=0$.

As a consequence of \eqref{eq:action}, we get that
\begin{align}
\label{eq:homogeneous}
\deg(x(m).v )  = \deg(v)  \quad \text{ and }
\quad \dep(x(m).v)= \dep(v) - m,
\end{align}
for
$m\geq 0$, $x\in\g$,  and any homogeneous element $v \in S(t^{-1}\g[t^{-1}])$ 
with respect to both gradations.

In the sequel, we will also
use the following notation, for $v$ of the form \eqref{eq:PBW_basis},
viewed either as an element of $V^k(\g)$ or of $S(t^{-1}\g[t^{-1}])$:
\begin{align}
\label{eq:deg_-1}
\deg_{-1}^{(0)}(v) :=
\sum\limits_{j=1}^\ell c_{j,1 },
 \end{align}
 which corresponds to the degree of the element
 obtained from $v^{(0)}$ by keeping only the terms of
 depth $0$, that is, the terms  $u^{i}(-1)$, $i=1,\ldots,\ell$.

Notice that a nonzero depth-homogenous
element of $S(t^{-1}\g[t^{-1}])$ has depth $0$ if
and  only if
its image in
$$R_{V^k(\g)} = V^k(\g)/t^{-2} \g[t^{-1}] V^k(\g)$$
is nonzero.
\section{Proof of the main result}
\label{sec:main_proof}
%\def\theequation{3.\arabic{equation}}
% Let $\g$ be a finite-dimensional simple Lie algebra over
% $\C$  with the standard non-degenerate  normalized invariant bilinear form on $\g$.
% For $k\in \C$ such that $k+h^{\vee}\neq 0$, where $h^{\vee}$ is the dual Coxeter number of $\g$,
% let $V^{k}(\g)$ be the associated affine  vertex operator algebra \cite{K,LL,FZ,Z}.

This section is devoted to the proof of Theorem \ref{Th:main}.

\subsection{Main strategy}
{Let $N_k$ be the maximal graded submodule of $V^k(\g)$,
so that $L_k(\g)=V^k(\g)/N_k$.
%Let $L_k(\g)$ be the unique simple graded quotient of $V^k(\g)$.
Our aim is to show that if $V^k(\g)$ is not simple, that is,
$N_k\not=\{0\}$, then $X_{L_k(\g)}$
is strictly contained in $\g^*\cong\g$, that is,
the image $I_k:=I_{N_k}$ of $N_k$ in
$R_{V^k(\g)}=\C[\g^*]$ is nonzero.}

%$R_{L_k(\g)}$ is a non trivial quotient of $\C[\g^*] \cong S(\g)$;
%see \S\ref{sub:Zhu affine vertex algebras}.

{
For $k=-h^\vee$, it follows from  \cite{FreGai04} that
$I_k$ is the defining ideal of the
nilpotent
cone $ \mathcal{N}(\g)$ of $\g$,
and so
$X_{L_k(\g)} = \mathcal{N}(\g)$ (see \cite{A11} or Subsection \ref{subsection:remark}
below).
Hence, there is no loss of generality in assuming that $k+h^\vee \not=0$.}

{
Henceforth, we suppose that $k+h^\vee \not=0$ and
that $V^k(\g)$ is not simple,
that is,
$N_k\ne \{0\}$.
Then there exists at least
one non-trivial (that is, nonzero and different from
{\bf 1}) singular vector $w$
in $V^k(\g)$.
Theorem \ref{theorem:image-of-the-singular}
states that
the image of $w$ in $I_k$ is nonzero,
and
this proves
Theorem \ref{Th:main}.
}

The rest of this section is devoted 
to the proof of  Theorem \ref{theorem:image-of-the-singular}.

Let $w$ be a nontrivial singular vector of $V^k(\g)$.
One can assume that $w \in F^p V^k(\g) \setminus F^{p+1} V^k(\g)$ for
some $p\in \Z_{\geq 0}$.

The image
$$\bar{w}:=\sigma(w)$$
of this singular vector in $S(t^{-1}\g[t^{-1}]) \cong {\rm gr}^F V^k(\g)$
is a nontrivial singular vector of $S(t^{-1}\g[t^{-1}])$.
Here $\sigma \colon V^k(\g) \to {\rm gr}^F V^k(\g)$
stands for the principal symbol map.
It follows from \eqref{eq:homogeneous}
that one can assume that $\bar{w}$ is homogenous with respect to both gradations
on $S(t^{-1}\g[t^{-1}])$.
In particular $\bar w$ has depth $p$.

%Observing\footnote{This is a general fact: for any vertex algebra $V$,
%we have $R_{V} = R_{{\rm gr}^F V}$.} that
%$$S(\g) \cong R_{V^k(\g)} = V^k(\g)/t^{-2} \g[t^{-1}] V^k(\g)
%%=S(t^{-1}\g[t^{-1}]) /t^{-2} \g[t^{-1}] S(t^{-1}\g[t^{-1}])
%= R_{S(t^{-1}\g[t^{-1}])},$$
%we see that
It is enough to show that
%the image of $w$ in $R_{S(t^{-1}\g[t^{-1}])}$ is nonzero, that is,
$p=0$, that is, $\bar w$ has depth zero.

%This is our goal
%for the rest of the section.
Write
$$w = \sum\limits_{j \in J} \lambda_j w^j,$$
where $J$ is a finite index set, $\lambda_j $ are nonzero
scalar for all $j \in J$,
and $w_j$ are pairwise distinct PBW monomials of the form \eqref{eq:PBW_basis}.
Let $I \subset J$ be the subset of $i \in J$ such that $\dep \bar w^{i} = p = \dep \bar w$.
Since $w \in F^p V^k(\g) \setminus F^{p+1} V^k(\g)$,
the set $I$ is nonempty.
%such that $\bar{w}^{i}$ have the same degree and the same depth as $\bar w$.
Here, $\bar{w}^{i}$ stands for the image of $w^{i}$ in ${\rm gr}^F V^k(\g) \cong S(t^{-1}\g[t^{-1}])$.

More specifically, for any $j \in I$, write
 \begin{align}
 \label{eq:PBW_w_i}
 w^{j} = ( w^{j})^{(+)} ( w^{j})^{(-)} ( w^{j})^{(0)}  {\bf 1},
 \end{align}
with
 \begin{align*}
& ( w^{j})^{(+)} := e_{\be_{1}}(-1)^{a_{1,1}^{(j)}}\cdots
e_{\be_{1}} (- r_1)^{a_{1,r_1}^{(j)}}\cdots
e_{\be_q}(- 1)^{a_{q,1}^{(j)}}\cdots e_{\be_q}(- r_q)^{a_{q,r_q}^{(j)}} \\\nonumber
& ( w^{j})^{(-)} : =f_{\be_1}(-1)^{b_{1,1}^{(j)}}\cdots f_{\be_1}(- s_1)^{b_{1,s_1}^{(j)}}\cdots
f_{\be_q}(- 1)^{b_{q,1}^{(j)}}
\cdots f_{\be_q}(- s_q)^{b_{q,s_q}^{(j)}} , \\\nonumber
&  ( w^{j})^{(0)} := u^1(- 1)^{c_{1,1}^{(j)}}\cdots u^1(- t_1)^{c_{1,t_1}^{(j)}}
\cdots u^\ell (- 1)^{c_{\ell,1}^{(j)}}\cdots u^\ell(- t_\ell )^{c_{\ell, t_\ell}^{(j)}},
\end{align*}
where $r_1,\ldots,r_q, s_1,\ldots,s_q,,t_1,\ldots,t_\ell$
are nonnegative integers, and
$a_{l,m}^{(j)},b_{l,n}^{(j)},c_{i,p}^{(j)}$, for $l =1,\ldots,q$,
$m=1,\ldots,r_l$,  $n=1,\ldots,s_l$,
$i =1,\ldots,\ell$,
$p=1,\ldots,t_i$,
are nonnegative integers such
that at least one of them is nonzero.

The integers $r_l$'s, for $l=1,\ldots,q$, are chosen so that at least one of the
$a_{l,r_l}^{(j)}$'s is nonzero for $j$ running through $J$
if for some $j \in J$, $(w^{j})^{(+)}\not= 1$.
Otherwise, we just set  $(w^{j})^{(+)} := 1$.
%If this choice is not possible, it means that
%there is no term in $u^{j}(-n)$, $n \in \Z_{>0}$, in $w^{i}$
%for all $i \in I$. In this case our convention is that $r_j=1$
%and $a_{1,r_1}^{(i)}=0$ for all $i \in I$.
Similarly are defined the  integers $s_l$'s and $t_m$'s,
for $l =1,\ldots,q$ and $m=1,\ldots,\ell$.

By our assumption, note that for all $i \in I$,
\begin{align*}
& \sum\limits_{n=1}^q \left( \sum\limits_{l=1}^{r_n}
a_{n,l}^{(i)} + \sum\limits_{l=1}^{s_n} b_{n,l}^{(i)} \right)
+  \sum\limits_{n=1}^\ell \sum\limits_{l=1}^{t_n} c_{n,l}^{(i)} = \deg(\bar w)&\\
&  \sum\limits_{n=1}^q \left( \sum\limits_{l=1}^{r_n}
a_{n,l}^{(i)} (l-1) + \sum\limits_{l=1}^{s_n} b_{n,l}^{(i)} (l-1) \right)
+ \sum\limits_{n=1}^\ell \sum\limits_{l=1}^{t_n} c_{n,l}^{(i)} (l-1)
= \dep(\bar w) = p.&
\end{align*}

\subsection{A  technical lemma}
\label{sub:technical}
In this paragraph we stand in the commutative setting, and we only
deal with $\bar w \in S(t^{-1}\g[t^{-1}])$ and its monomials $\bar w^{i}$'s,
for $ i \in I$.

Recall from \eqref{eq:deg_-1} that,
$$
\deg_{-1}^{(0)}(w^i)= \sum\limits_{j=1}^\ell c_{j,1}^{(i)}
$$
for $i \in I$.
Set
$$d_{-1}^{(0)}(I) := \max\{\deg_{-1}^{(0)}( w^i) \colon i \in I\},$$
and
$$I_{-1}^{(0)} := \{ i\in I \colon \deg_{-1}^{(0)}(w^i)= d_{-1}^{(0)}(I)\}.$$
If $(w^i)^{(0)}=1$ for all $i \in I$, we just set $d_{-1}^{(0)}(I)=0$
and then $I_{-1}^{(0)}=I$.

\begin{lem}
\label{lem:wi_-}
If  $i\in I_{-1}^{(0)}$, then $(\bar w^i)^{(-)}=1$. In
 other words, for $i \in I_{-1}^{(0)}$, we have
$\bar w^{i}= (\bar w^{i})^{(0)}  (\bar w^{i})^{(+)} {\bf 1}.$
\end{lem}
\begin{proof} Suppose the assertion is false.
Then for some positive roots $\beta_{j_1},\ldots,\beta_{j_t} \in \Delta_+$,
one can write for any $i \in I_{-1}^{(0)}$,
\begin{align}\label{e5}
(\bar w^i)^{(-)}=f_{\be_{j_1}}(-1)^{b_{j_1,1}^{(i)}}
\cdots f_{\be_{j_1}}(-s_{j_1})^{b_{j_1,s_{j_1}}^{(i)}}\cdots
f_{\be_{j_t}}(-1)^{b_{j_t,1}^{(i)}}
\cdots f_{\be_{j_t}}(-s_{j_t})^{b_{j_t,s_{j_t}}^{(i)}},
\end{align}
so that for any $l \in \{1,\ldots,t\}$,
$$\{b_{j_l,s_{j_l}}^{(i)}\colon i \in I_{-1}^{(0)}\}\neq \{0\}.$$
Set
$$
K_{-1}^{(0)} = \{i \in I_{-1}^{(0)} \colon b_{j_1,s_{j_1}}^{(i)}> 0 \}.$$
Since $\bar w$ is a singular vector of $S(t^{-1}\g[t^{-1}])$ and $s_{j_1}-1 \in\Z_{\geq 0}$,
we have $e_{\beta_{j_1}}(s_{j_1}-1) .\bar w=0$.
On the other hand, using the action of $\g[t]$ on $S(t^{-1}\g[t^{-1}])$
as described by \eqref{eq:action}, we see that
\begin{align}
\label{eq:wi_-}
0=e_{\beta_{j_1}}(s_{j_1}-1) .\bar w  =
\sum\limits_{i \in K_{-1}^{(0)}} \lambda_ib_{j_1,s_{j_1}}^{(i)} v^{i} + v,
\end{align}
where for $i \in K_{-1}^{(0)}$,
\begin{align*}
v^{i} & :=
(\bar w^{i})^{(0)} {\be_{j_1}(-1)}
f_{\be_{j_1}}(-1)^{b_{j_1,1}^{(i)}}
\cdots f_{\be_{j_1}}(-s_{j_1})^{b_{j_1,s_{j_1}}^{(i)}-1} & \\
& \qquad \quad \cdots
f_{\be_{j_t}}(-1)^{b_{j_t,1}^{(i)}}
\cdots f_{\be_{j_t}}(-s_{j_t})^{b_{j_t,s_{j_t}}^{(i)}} (w^{i})^{(+)} {\bf 1}, &
\end{align*}
and $v$ is a linear combination of monomials $x$
such that
$$\deg_{-1}^{(0)}(x) \leq d_{-1}^{(0)}(I).$$
Indeed, for $i \in K_{-1}^{(0)}$, it is clear that
\begin{align*}
e_{\beta_{j_1}}(s_{j_1}-1) .w^{i}  = b_{j_1,s_{j_1}}^{(i)} v^{i} + y^{i},
\end{align*}
where $y^{i}$ is a linear combination of monomials $y$
such that
$\deg_{-1}^{(0)}(y) \leq d_{-1}^{(0)}(I)$ because
${\rm ht}(\beta_{j_1}) \leq {\rm ht}(\beta_{j_l})$
for all $l \in \{1,\ldots,t\}$.
Next, for $i \in I_{-1}^{(0)} \setminus K_{-1}^{(0)}$,
$e_{\beta_{j_1}} (s_{j_1}-1).\bar w^{i} $ is a
linear combination of monomials $z$ such that
$\deg_{-1}^{(0)}(z) \leq d_{-1}^{(0)}(I)$ because
$b_{j_1,s_{j_1}}^{(i)}=0$.
Finally, for $i \in I \setminus I_{-1}^{(0)}$, we have $\deg_{-1}^{(0)}(\bar w^{i}) < d_{-1}^{(0)}(I) $
and, hence, $e_{\beta_{j_1}}(s_{j_1}-1) .\bar w^{i} $ is a
linear combination of monomials $z$ such that
$\deg_{-1}^{(0)}(z) \leq d_{-1}^{(0)}(I)$ as well.

Now, note that for each $i \in K_{-1}^{(0)}$,
%$$(v^{i})^{(+)} = (w^{i})^{(+)},\quad
%\deg \,(v^{i})^{(-)} = \deg\, (w^{i})^{(-)} - 1$$
%and
$$\deg_{-1}^{(0)}(v^{i})=\deg_{-1}^{(0)}(\bar w^{i})+1 = d_{-1}^{(0)}(I) +1.$$
Hence by \eqref{eq:wi_-} we get a contradiction because all monomials
$v^{i}$, for $i$ running through $K_{-1}^{(0)}$, are linearly independent
while $\lambda_i b_{j_1,s_{j_1}}^{(i)} \not=0$, for $i\in K_{-1}^{(0)}$.
This concludes the proof of the lemma.
\end{proof}

\subsection{Use of Sugawara operators}
\label{sub:Sugawara operators}
Recall that $w = \sum\limits_{j \in J} \lambda_j w^j$.
Let $J_1\subseteq J$ be such that for $i\in J_1$, $(w^i)^{(-)}=1$. Then by Lemma \ref{lem:wi_-},
$$%\varnothing
 \emptyset \neq I_{-1}^{(0)}\subseteq J_1.$$ So $J_1\neq \emptyset.$
Set
$${d}_{-1}^{(0)}: = {d}_{-1}^{(0)} (J_1) = \max\{\deg_{-1}^{(0)}( w^i) \colon i \in J_1\},$$
and
$$J_{-1}^{(0)} := \{ i\in J_1 \colon \deg_{-1}^{(0)}(w^i)= {d}_{-1}^{(0)}\}.$$
Then $d_{-1}^{(0)}(I)\leq {d}_{-1}^{(0)}$. Set
$$d^{+} := \max\{\deg\,(w^{i})^{(+)} \colon i \in J_{-1}^{(0)} \}$$
and let
$$J^+= \{i \in J_{-1}^{(0)} \colon \deg \,( w^{i})^{(+)}= d^{+} \}\subseteq J_{-1}^{(0)}.$$
Our next aim is to show that for  $i\in J^+$,
$w^{i}$ has depth zero, whence $p=0$ since $p$ is by definition
the smallest depth of the $w^j$'s, and so
the image of $w$ in
$R_{V^k(\g)} = F^0 V^k(\g)/ F^1 V^k(\g)$ is nonzero.

This will be achieved in this paragraph through the use of
the Sugawara construction.

Recall that
by Lemma \ref{lem:Sugawara_singular_vector},
$$
L_{-1} w=\tilde{L}_{-1} w$$
since $w$ is a singular vector of $V^k(\g)$, where
\begin{align*}
{\tilde{L}_{-1}:=\dfrac{1}{k+h^\vee} \left(\sum\limits_{i=1}^{\ell}u^i(-1)u^i(0)
+\sum\limits_{\al\in{\Delta}_{+}}e_{\al}(-1)f_{\al}(0)\right)}.
\end{align*}
%On the other hand, note that $L_{-1}$
%is nothing but the translation operator $T$ of the conformal vertex algebra
%$V^k(\g)$.

\begin{lem}
\label{lem:Sugawara_monomial}
Let $z$ be a PBW monomial of the form \eqref{eq:PBW_basis}.
Then $\tilde L_{-1} z$
is a linear combination of of PBW monomials
$x$ satisfying all the following conditions:
\begin{enumerate}
\item[(a)] $\deg(x^{(+)}) \leq \deg(z^{(+)})+1$ and $\deg(x^{(0)}) \leq \deg(z^{(0)})+1$,
\item[(b)] if $z^{(-)}\neq 1$, then $x^{(-)}\neq 1$.
\item[(c)] if $x^{(-)}=z^{(-)}$, then either $\deg(x^{(0)}) =\deg(z^{(0)}) +1$, or $x^{(0)} = z^{(0)}$.
\item[(d)] if $\deg(x^{(0)}) =\deg(z^{(0)}) +1$, then $x^{(-)}= z^{(-)}$ and $\deg(x^{(+)}) \leq \deg(z^{(+)})$.
\end{enumerate}
\end{lem}
\begin{proof}
Parts (a)--(c) are easy to see. We only prove (d).

Assume that $\deg(x^{(0)}) =\deg(z^{(0)}) +1$.
Either $x$ comes from the term $\sum\limits_{i=1}^{\ell} u^i(-1)u^i(0) z$,
or it comes from a term $e_{\al}(-1)f_{\al}(0)z$ for some $\al \in \Delta_+$.

If $x$ comes from the term $\sum\limits_{i=1}^{\ell} u^i(-1)u^i(0) z$, then it is obvious that
$x^{(-)}= z^{(-)}$ and $x^{(+)} = z^{(+)}$.

Assume that $x$ comes from $e_{\al}(-1)f_{\al}(0)z$ for some $\al \in \Delta_+$.
We have

\begin{align*}
e_{\al}(-1)f_{\al}(0)z %& =e_{\al}(-1)f_{\al}(0)  v^{(+)} v^{(-)}v^{(0)} & \\
& = e_{\al}(-1)  [f_{\al}(0), z^{(+)}]z^{(-)}z^{(0)}{\bf 1} + e_{\al}(-1)z^{(+)}[f_{\al}(0), z^{(-)}]z^{(0)} {\bf 1} & \\
& \quad + e_{\al}(-1)z^{(+)} z^{(- )}[f_{\al}(0) ,z^{(0)} ]  {\bf 1}. &
\end{align*}
Clearly, any PBW monomials $x$ from
$$e_{\al}(-1)z^{(+)}[f_{\al}(0), z^{(-)}]z^{(0)} {\bf 1} \quad \text{ or }
\quad e_{\al}(-1)z^{(+)} z^{(- )}[f_{\al}(0) ,z^{(0)} ]  {\bf 1}$$
satisfies that $\deg(x^{(0)}) \leq \deg(z^{(0)})$. Then  it is enough to consider PBW monomials in
$$e_{\al}(-1)  [f_{\al}(0), z^{(+)}]z^{(-)}z^{(0)}{\bf 1}.$$
The only possibility for a PBW monomial $x$ in $ e_{\al}(-1)  [f_{\al}(0), z^{(+)}]z^{(-)}z^{(0)}{\bf 1}$
to verify $\deg(x^{(0)}) =\deg(z^{(0)}) +1$ is that it comes from a term
$ [f_{\al}(0), e_{\al}(-n)]= - \al(-n)$ for some $n \in \Z_{>0}$, where
$e_{\al}(-n)$ is a term in $z^{(+)}$.
 But then, for PBW monomials $x$ in
 $ e_{\al}(-1)[f_{\al}(0), z^{(+)}]z^{(0)}{\bf 1}$
 such that $\deg(x^{(0)}) =\deg(z^{(0)}) +1$,
we have $x^{(-)}=z^{(-)}$ and $\deg(x^{(+)}) \leq \deg(z^{(+)})$.
\end{proof}

We now consider the action of $\tilde{L}_{-1}$  on particular PBW monomials.
\begin{lem}
\label{lem:Sugawara_bracket}
Let $z$ be a PBW monomial of the form \eqref{eq:PBW_basis}
such that $z^{(-)}=1$ and $\dep(z^{(+)})=0$, that is,
either $z^{(+)}=1$, or for some $j_1,\ldots,j_t \in \{1,\ldots,q\}$
(with possible repetitions),
\begin{align*}
z = e_{\be_{j_1}}(-1)e_{\be_{j_2}}(-1)\ldots e_{\be_{j_t}}(-1) z ^{(0)} {\bf 1}.
\end{align*}
Then $\tilde{L}_{-1}z$
%\begin{align*}
%& \left (\sum\limits_{i=1}^lu^i(-1)u^i(0)+\sum\limits_{i=1}^qe_{\be_i}(-1)f_{\be_i}(0) \right).v \\
%= & \sum\limits_{l=1}^t b_{l}  v^{(0)} e_{\be_{j_1}}(-1)\cdots e_{\be_{j_{l-1}}}(-1)e_{\be_{j_l}}(-2)
%e_{\be_{j_{l+1}}}(-1)\cdots e_{\be_{j_t}}(-1){\bf 1} + x,
%\end{align*}
%for some $b_{l}\in \C$, where $x$
 is a linear combination of PBW monomials $y$ satisfying one of the following
 conditions:
 \begin{enumerate}
\item $y^{(-)}=1$, $\dep(y^{(+)})\geq 1$, $\deg(y^{(+)})\leq \deg(z^{(+)})$,
$y^{(0)}=z^{(0)}$,
 % \item $\dep(y^{(+)}) =0$, $y^{(-)}=1$, $\deg (y^{(+)}) = \deg (z^{(+)})$ and
% $\deg_{-1}^{(0)}(y) >\deg_{-1}^{(0)}(z)$,
{\item $y^{(-)}=1$, $\dep(y^{(+)})=0$,  $\deg (y^{(+)})\leq \deg (z^{(+)})-1$, and
$\deg(y^{(0)})>\deg(z^{(0)})$, $\deg_{-1}^{(0)}(y)=\deg_{-1}^{(0)}(z)$,}
{\item $y^{(-)}=1$, $\dep(y^{(+)})\geq 1 $, $\deg(y^{(+)})\leq \deg(z^{(+)})-1$, and $\deg_{-1}^{(0)}(y)=\deg_{-1}^{(0)}(z)+1$,}
\item $y^{(-)}\neq 1$.
\end{enumerate}
\end{lem}
\begin{proof}
First, we have
\begin{align*}
\sum\limits_{i=1}^{\ell} u^i(-1)u^i(0) z &
=   \sum_{r=1}^t
e_{\be_{j_1}}(-1) \ldots [\sum\limits_{i=1}^{\ell}u^i(-1)u^i(0), e_{\be_{j_r}}(-1)] \ldots e_{\be_{j_t}}(-1)  z^{(0)} {\bf 1},&
\end{align*}
and
\begin{align*}
 [\sum\limits_{i=1}^{\ell}u^i(-1)u^i(0), e_{\be_{j_r}}(-1)] & = \sum\limits_{i=1}^{\ell}\left(u^i(-1)  [u^i(0), e_{\be_{j_r}}(-1)] + [u^i(-1), e_{\be_{j_r}}(-1)] u^i(0) \right)&\\
 & = \be_{j_r}(-1) e_{\be_{j_r}}(-1) + e_{\be_{j_r}}(-2) \be_{j_r}(0) .&
\end{align*}
So
\begin{align}\label{e37}
& \sum\limits_{i=1}^{\ell} u^i(-1)u^i(0) z &\\  \nonumber
& = \sum_{r=1}^t
e_{\be_{j_1}}(-1) \ldots (\be_{j_r}(-1) e_{\be_{j_r}}(-1) + e_{\be_{j_r}}(-2) \be_{j_r}(0)) \ldots e_{\be_{j_t}}(-1)  z^{(0)} {\bf 1}.
\end{align}
Second, we have
\begin{align*}
\sum\limits_{\al\in{\Delta}_+}e_{\al}(-1) f_{\al}(0) z &
= \sum\limits_{\al\in{\Delta}_+}  \sum_{r=1}^te_{\al}(-1)
e_{\be_{j_1}}(-1) \ldots [f_{\al}(0)  ,e_{\be_{j_r}}(-1)] \ldots e_{\be_{j_t}}(-1) z^{(0)}  {\bf 1} & \\
 & \qquad +  \sum\limits_{\al\in{\Delta}_+} e_{\al}(-1)e_{\be_{j_1}}(-1)e_{\be_{j_2}}(-1)\ldots e_{\be_{j_t}}(-1)
 [f_{\al}(0),z^{(0)} ] {\bf 1} .&
\end{align*}
It is clear that
 any PBW monomial $y$ in
$$\sum\limits_{\al\in{\Delta}_+}e_{\al}(-1) e_{\be_{j_1}}(-1)e_{\be_{j_2}}(-1)\ldots e_{\be_{j_t}}(-1)
 [f_{\al}(0),z^{(0)} ] {\bf 1}$$ verifies that
 \begin{align}\label{e38}
 y^{(-)} \not= 1.
\end{align}
We now consider
$$u_r:=\sum\limits_{\al\in{\Delta}_+}
e_{\al}(-1)e_{\be_{j_1}}(-1) \ldots [f_{\al}(0)  ,e_{\be_{j_r}}(-1)] \ldots e_{\be_{j_t}}(-1) z^{(0)}  {\bf 1}, \text{ for } 1\leq r\leq t.$$

$\ast$ If $\be_{j_r}=\al+\be$ for some $\al,\be\in\Delta_{+}$, then  there is a partial sum  of two terms in $u_r$:
\begin{align*}
& c_{-\al,\al+\be}e_{\al}(-1)e_{\be_{j_1}}(-1) \ldots e_{\be}(-1)\ldots e_{\be_{j_t}}(-1) z^{(0)}  {\bf 1} &\\ & +c_{-\be,\al+\be}e_{\be}(-1)e_{\be_{j_1}}(-1) \ldots e_{\al}(-1)\ldots e_{\be_{j_t}}(-1) z^{(0)}  {\bf 1}.
\end{align*}

Rewriting  the above sum to a linear combination of PBW monomials,  and noticing that
$$
c_{-\al,\al+\be}e_{\al}(-1)e_{\be}(-1)+c_{-\be,\al+\be}e_{\be}(-1)e_{\al}(-1)=c_{-\al,\al+\be}c_{\al,\be}e_{\al+\be}(-2),
$$
due to \eqref{eq:structure_constant},
we deduce that it is  a linear combination of PBW monomials  $y$ such that
\begin{align}\label{e43}y^{(-)}=z^{(-)}=1, \ y^{(0)}=z^{(0)}, \ \dep(y^{(+)})\geq 1, \ \deg(y^{(+)})\leq \deg(z^{(+)}),
\end{align}
where $c_{-\al,\al+\be}, c_{-\be,\al+\be}, c_{\al,\be}\in {\mathbb R}^*$.

$\ast$ If $\al-\be_{j_r}\in\Delta_+$ for some $\al\in\Delta_+$, then there is a term in $u_r$:
\begin{align}\label{e35}
c_{-\al,\be_{j_r}}e_{\al}(-1)e_{\be_{j_1}}(-1) \ldots e_{\be_{j_{r-1}}}(-1) f_{\al-\be_{j_r}}(-1)e_{\be_{j_{r+1}}}(-1) \ldots e_{\be_{j_t}}(-1) z^{(0)}  {\bf 1}.
\end{align}
It is easy to see  that (\ref{e35}) is a linear combination of PBW monomials $y$ such that $y$ satisfies one of the following:
\begin{align}\label{e39}
& y^{(-)}=1, \ \dep(y^{(+)})\geq 1, \ \deg(y^{(+)})\leq \deg(z^{(+)}), \ y^{(0)}=z^{(0)},&\\
& y^{(-)}=1, \ \dep(y^{(+)})=0, \ \deg(y^{(+)})\leq \deg(z^{(+)})-1, &\\ \nonumber & \deg(y^{(0)})>\deg(z^{(0)}), \ \deg_{-1}^{(0)}(y)=\deg_{-1}^{(0)}(z),&\\
& y^{(-)}\neq 1.
\end{align}
Notice also that with $\al=\be_{j_r}$,  there is a term in $u_r$:
\begin{align*}
-e_{\be_{j_r}}(-1)e_{\be_{j_1}}(-1) \ldots e_{\be_{j_{r-1}}}(-1)
\be_{j_r}(-1)e_{\be_{j_{r+1}}}(-1) \ldots e_{\be_{j_t}}(-1) z^{(0)}  {\bf 1}.
\end{align*}
Together with (\ref{e37}), we see that
\begin{align*}
& \sum\limits_{i=1}^{\ell} u^i(-1)u^i(0) z+\sum\limits_{r=1}^te_{\be_{j_r}}(-1)e_{\be_{j_1}}(-1)\ldots [f_{\be_{j_r}}(0), e_{\be_{j_r}}(-1)]\ldots \ldots e_{\be_{j_t}}(-1) z^{(0)}  {\bf 1}&\\
=& \sum_{r=1}^t
e_{\be_{j_1}}(-1) \ldots (\be_{j_r}(-1) e_{\be_{j_r}}(-1) + e_{\be_{j_r}}(-2) \be_{j_r}(0)) \ldots e_{\be_{j_t}}(-1)  z^{(0)} {\bf 1}&\\
& -\sum\limits_{r=1}^t\sum\limits_{s=1}^{r-1}e_{\be_{j_1}}(-1) \ldots [e_{\be_{j_r}}(-1), e_{\be_{j_s}}(-1)]\ldots e_{\be_{j_{r-1}}}(-1)\be_{j_r}(-1)e_{\be_{j_{r+1}}}(-1) \ldots e_{\be_{j_t}}(-1) z^{(0)}  {\bf 1}&\\
& -\sum\limits_{r=1}^te_{\be_{j_1}}(-1) \ldots e_{\be_{j_{r-1}}}(-1)e_{\be_{j_r}}(-1)\be_{j_r}(-1)e_{\be_{j_{r+1}}}(-1) \ldots e_{\be_{j_t}}(-1) z^{(0)}  {\bf 1}&
\end{align*}
is a linear combination of PBW monomials $y$ satisfying one of the following:
\begin{align}
\label{e40}
& y^{(-)}=1, \ \dep(y^{(+)})\geq 1, \deg(y^{(+)})\leq \deg(z^{(+)}), \ y^{(0)}=z^{(0)}, \\\label{e41}
& y^{(-)}=1, \ \dep(y^{(+)})\geq 1, \deg(y^{(+)})\leq \deg(z^{(+)})-1, \\
& \ \deg_{-1}^{(0)}(y)=\deg_{-1}^{(0)}(z)+1. \nonumber
\end{align}
Then the lemma follows from (\ref{e38}), (\ref{e43}), (\ref{e39})--(\ref{e41}).
\end{proof}

\begin{lem}\label{Sugawara1}
Let $z$ be a PBW monomial of the form \eqref{eq:PBW_basis}
such that $z^{(-)}=1$. Then
\begin{align*}
& \tilde{L}_{-1}z
= c z^{(+)}(\gamma-\sum\limits_{j=1}^qa_{j,1}\be_j)(-1)z^{(0)}+y^1,
\end{align*}
where $c$ is a nonzero constant,
$\gamma=\sum\limits_{j=1}^q\sum\limits_{s=1}^{r_j}a_{j,s}\be_j$, and
$y^1$ is a linear combination of PBW monomials $y$ such that
$$
\deg_{-1}^{(0)}(y)=\deg_{-1}^{(0)}(z)+1, \ \deg(y^{(+)})\leq \deg(z^{(+)})-1,
$$
or
$$\deg_{-1}^{(0)}(y)\leq \deg_{-1}^{(0)}(z).$$
\end{lem}
\begin{proof}
Since the proof is similar to that of Lemma \ref{lem:Sugawara_bracket},
we left the verification to the reader.
\end{proof}

{
\begin{lem}\label{lem:wi_+}
 For  $i\in J^+$, we have that $\dep((w^i)^{(+)})=0$.
\end{lem}
\begin{proof} First we have
$$
w = \sum\limits_{j \in J^+} \lambda_j w^j+\sum\limits_{j \in J_{-1}^{(0)}\setminus J^+} \lambda_j w^j+\sum\limits_{j \in J_1\setminus J_{-1}^{(0)}} \lambda_j w^j+\sum\limits_{j \in J\setminus J_1} \lambda_j w^j.
$$
Then by (b) of Lemma \ref{lem:Sugawara_monomial} and Lemma \ref{Sugawara1}, we have
\begin{align*}
& (k+h^{\vee})\tilde{L}_{-1}w &\\
=& \sum\limits_{i\in J^+}(w^i)^{(+)} {(\gamma_i-\sum\limits_{j=1}^q
a_{j,1}^{(i)}\be_j)(-1)}(w^i)^{(0)}+\sum\limits_{i\in J_1\setminus J^+}(w^i)^{(+)}(\gamma_i-\sum\limits_{j=1}^qa^{(i)}_{j1}\be_j)(-1)(w^i)^{(0)}+y^1,
\end{align*}
where
{$\gamma_i=\sum\limits_{j=1}^q\sum\limits_{s=1}^{r^{(i)}_j}a^{(i)}_{j,s}\be_i$},
for $i\in J_1$, and $y^1$ is a linear combination of PBW monomials $y$ satisfying
one of the following conditions:
\begin{align*}
& \deg_{-1}^{(0)}(y)=d_{-1}^{(0)}+1, \ \deg(y^{(+)})\leq d^+-1,\\
& \deg_{-1}^{(0)}(y)\leq d_{-1}^{(0)},\\
& y^{(-)}\neq 1. 
\end{align*}
 On the other hand,  by Lemma \ref{lem:Sugawara_singular_vector}
 $$L_{-1}w=\tilde{L}_{-1}w.$$
By Lemma \ref{lem:sugawara-vs-g}, there is no PBW monomial $y$ in $L_{-1}w$ such that $\deg(y^{(+)})=d^+$, $y^{(-)}=1$, and $\deg_{-1}^{(0)}(y)=d_{-1}^{(0)}+1$. Then we deduce that
$$
\sum\limits_{i\in J^+}(w^i)^{(+)}{(\gamma_i-\sum\limits_{j=1}^qa_{j,1}^{(i)}\be_j)(-1)}(w^i)^{(0)}=0,
$$
which means that ${(\gamma_i-\sum\limits_{j=1}^qa_{j,1}^{(i)}\be_j)}=0$, for $i\in J^+$, that is, $\dep((w^i)^{(+)})=0$.
\end{proof}}

As explained at the beginning of \S\ref{sub:Sugawara operators}, Theorem
\ref{Th:main} will be a consequence of the following lemma.

\begin{lem}
\label{lem:w_0}
For each $i\in J^+$, we have $\dep(w^{i})=0$.
\end{lem}
\begin{proof}
By definition, for $i\in J^+$,  $(w^i)^{(0)}=1$. Moreover, by Lemma \ref{lem:wi_+},
$\dep((w^i)^{(+)})=0$. Hence it suffices to prove that for $i\in J^+$,
$$(w^{i})^{(0)}=u^{1}(-1)^{c^{(i)}_{1,1}}\cdots u^\ell (-1)^{c^{(i)}_{\ell,1}}.$$
%The fact that $\dep(w)=0$ if $\dep(w^{(i)})=0$ for some $i \in I$
%was already observed at the end of \S\ref{sub:technical}.
%Now by Lemma \ref{lem:wi_-} and Lemma \ref{lem:wi_+},
%for $i \in I^+$, $\dep(w^{(i)})=\dep((w^{(i)})^{(0)})$.
%Hence it suffices to show that for $i\in I^+$, $\dep((w^{(i)})^{(0)})=0$.
Suppose the contrary. Then
there exists $i \in J^+$ such that
\begin{align*}
w^i = & e_{\be_1}(-1)^{a^{(i)}_{1,1}}\cdots e_{\be_{q}}(-1)^{a^{(i)}_{q,1}}  u^{1}(-1)^{c^{(i)}_{1,1}}
\cdots u^{1}(-m_{1})^{c^{(i)}_{1,m_1}} & \\
& \qquad \quad
\cdots u^\ell (-1)^{c^{(i)}_{\ell,1}} \cdots u^\ell(-m_{\ell})^{c^{(i)}_{\ell,m_\ell}} {\bf 1},
\end{align*}
with at least one of the $m_j$'s, for $j=1,\ldots,\ell$, strictly greater than $1$
and $c^{(i)}_{j,m_j}  \not=0$ for such a $j$.
Without loss of generality,
one may assume that $1 \in J^+$, that
 $$m_1=\max\{ m_j\colon j=1,\ldots,\ell\} \quad \text{ and that }
 \quad 0\neq c^{(1)}_{1,m_1} \geq c^{(i)}_{1,m_1},  \text{ for } i\in J^+.$$
Writing $L_{-1}w$ as
$$L_{-1}w=\sum\limits_{i \in J^+ }L_{-1}w^i
+\sum\limits_{i\in J_{-1}^{(0)}\setminus J^+}L_{-1}w^i
+\sum\limits_{i\in J_1\setminus J_{-1}^{(0)}}L_{-1}w^i
+\sum\limits_{i \in J \setminus J_1}L_{-1}w^i,
$$
we see by Lemma \ref{lem:sugawara-vs-g} that
\begin{align}\label{e17}
L_{-1}w=\lambda_1 m_1c^{(1)}_{1,m_1} v^1+\sum\limits_{i\in J^+,i\neq 1}\lambda_im_1c^{(i)}_{1,m_1}v^i+v
+v',
\end{align}
where for $i\in J^+$, $v^{i}$ is the PBW monomial defined by:
\begin{align}\label{e14}
(v^i)^{(-)}=(w^i)^{(-)}=1,
\end{align}
\begin{align}
\label{e13}
(v^i)^{(+)}=(w^i)^{(+)}=e_{\be_1}(-1)^{a^{(i)}_{1,1}}\cdots e_{\be_{q}}(-1)^{a^{(i)}_{q,1}},
\end{align}
\begin{align}\label{e33}
(v^i)^{(0)} & = u^{1}(-1)^{c^{(i)}_{1,1}}\cdots
  u^{1}(-m_{1})^{c^{(i)}_{1,m_1}-1}u^{1}(-m_{1}-1)\cdots  u^\ell (-m_{\ell})^{c^{(i)}_{\ell,m_\ell}},&
\end{align}
and so, by definition of $J^+ \subset J_{-1}^{(0)}$,
\begin{align}\label{e34}
\deg_{-1}^{(0)}(v^i)=d_{-1}^{(0)},
\end{align}
$v$ is a linear combination of PBW monomials $x$ such that
$$
x^{(0)}=u^{1}(-1)^{c^{(x)}_{1,1}}\cdots u^{1}(- n_{1}^{(x)})^{c^{(x)}_{1,n_{1}^{(x)}}}
\cdots u^\ell (-1)^{c^{(x)}_{\ell,1}}\cdots u^\ell(-n_{\ell}^{(x)})^{c^{(x)}_{\ell,n_{\ell}^{(x)}}}
$$
and either,
\begin{align*}
n_{1}^{(x)} \leq m_1,
\end{align*}
or
\begin{align*}
& \deg(x^{(+)}) \leq d^+ -1,
\end{align*}
or
\begin{align*}
& \deg_{-1}^{(0)}(x )\leq d_{-1}^{(0)}-1,
\end{align*}
and $v'$ is a linear combination of PBW monomials $x$ such that $x^{(-)}\neq 1$.
Note that the assumption that $m_1\geq 2$ makes sure that (\ref{e34}) holds,
and that $\dep(v^{i}) =\dep(w^i)+1$ for all $i \in J^+$.

On the other hand, by Lemma \ref{lem:Sugawara_singular_vector},
\begin{align*}
L_{-1} w={\tilde{L}_{-1} w},
%\label{eq:sugawara-action}
\end{align*}
%where
%\begin{align*}
%{\tilde{L}_{-1}=\dfrac{1}{k+h^\vee} \left(\sum\limits_{i=1}^{\ell}u^i(-1)u^i(0)
%+\sum\limits_{\al\in{\Delta}_{+}}e_{\al}(-1)f_{\al}(0)\right)},
%\end{align*}
since $w$ is a singular vector of $V^k(\g)$.
Hence $v^1$ must be a PBW monomial of {$\tilde{L}_{-1} w$}.
Our strategy to obtain the expected contradiction
is to show that there is no PBW monomial $v^1$ in {$\tilde{L}_{-1} w^i$}
for each $i \in J$.

%Note that by \eqref{e14} and \eqref{e13}, $v^1$ is a monomial
%in $L_{-1} w$ such that $(v^1)^{(-)}=1$
%and $\dep((v^1)^{(+)})=0$.
%Hence by Lemma~\ref{lem:Sugawara_bracket} and \eqref{e13}
%we get that if $v^1$ is a monomial
%in $L_{-1} w^{i}$ for some $i \in I_{-1}^{(0)}$,
%$$d^+ = \deg ((v^1)^{(+)}) < \deg ((w^{i})^{(+)}),$$
%%\begin{align} \label{eq:v1-1}
%%& d^+ = \deg ((v^1)^{(+)}) < \deg ((w^{i})^{(+)}),& \\\label{eq:v1-1}
%%& \deg((v^1)^{(0)}) = \deg((w^{i})^{(0)}), & \\ \label{eq:v1-1}
%%$ d_{-1}^0=d_{-1}^{(0)}(v^1)= d_{-1}^{(0)}(w^{i}), &
%%\end{align*}
%whence a contradiction since $d^+$ is the maximum of
%the $\deg ((w^{i})^{(+)})$, for $i$ running through $I$.
%If $i \not\in I_{-1}^{(0)}$.
%By \eqref{eq:v1-1}, $v^1$ cannot be a monomial of $w^{i}$ for
$\ast$ Assume that $i\in J^+$, and suppose that
$v^{1}$ is a PBW monomial  in ${\tilde{L}_{-1} w^{i}}$.
First of all, $\deg((w^{i})^{(+)}) = d^+$
because $i \in J^+$. Moreover,
 by the definition of $J_1$ and Lemma~\ref{lem:wi_+},
we have $(w^{i})^{(-)}=1$ and $\dep((w^{i})^{(+)})=0$.
Hence by (2) of Lemma~\ref{lem:Sugawara_bracket},
%Then by  (\ref{e14}) and (\ref{e33}), $(v^1)^{(-)}=1$ and $(v^1)^{(0)}\neq (w^i)^{(0)}$, and
%so by Lemma~\ref{lem:Sugawara_bracket},
$$\deg((v^1)^{(+)}) < \deg((w^{i})^{(+)}) = d^+$$
because $(v^1)^{(-)}=1$ and $\dep((v^{1})^{(+)})=0$ by \eqref{e14}
and \eqref{e13}.
But $d^+ = \deg((v^1)^{(+)})$ by \eqref{e13},
 whence a contradiction.

$\ast$ Assume that $i \in J_{-1}^{(0)} \setminus J^+$.
By the definition of $J^+$ and \eqref{e13},
\begin{equation}
\label{eq:d+}
\deg ((w^{i})^{(+)}) < d^+ =\deg((v^1)^{(+)}).
\end{equation}
Suppose that
$v^{1}$ is a PBW monomial  in ${\tilde{L}_{-1} w^{i}}$.
Then
\begin{equation}
\label{e15}
(w^{i})^{(-)}=1 = (v^1)^{(-)}
\end{equation}
by Lemma \ref{lem:wi_-} since $i \in J_{-1}^{(0)}$.
The last equality follows from \eqref{e14}.
Then by (c) of Lemma~\ref{lem:Sugawara_monomial},
either $\deg((v^1)^{(0)})=\deg((w^{i})^{(0)})+1$,
or $( v^1)^{(0)}=(w^{i})^{(0)}.$
But it is impossible that $\deg((v^1)^{(0)})=\deg((w^{i})^{(0)})+1$,
by (d) of Lemma \ref{lem:Sugawara_monomial} because $\deg((v^1)^{(+)}) > \deg ((w^{i})^{(+)})$.
Therefore,
$$
(v^1)^{(0)}=(w^{i})^{(0)}.$$
Computing $\tilde{L}_{-1} w^{i}$, we deduce from
$$
(v^1)^{(+)}=e_{\be_1}(-1)^{a^{(1)}_{1,1}}\cdots e_{\be_{q}}(-1)^{a^{(1)}_{q,1}},
$$
that
$$
(w^i)^{(+)}=e_{\be_1}(-1)^{a^{(j)}_{1,1}}\cdots e_{\be_{q}}(-1)^{a^{(j)}_{q,1}}.
$$
{Since $(v^1)^{(-)}=(w^{i})^{(-)}=1$, it
results from Lemma \ref{lem:Sugawara_bracket}
that $\deg((v^1)^{(+)})\leq \deg((w^{i})^{(+)})$,
which contradicts \eqref{eq:d+}.}

$\ast$ Assume that $i\in J_1\setminus J_{-1}^{(0)}$.
Then
 \begin{align}
 \label{eq:d01}
 \deg_{-1}^{(0)} (w^{i})< d_{-1}^{(0)} =\deg_{-1}^{(0)}(v^1)
 \end{align}
 by \eqref{e34}.
 Suppose that  $v^1$ is a PBW monomial  in ${\tilde{L}_{-1} w^{i}}$.
 By (b) and (c) of Lemma~\ref{lem:Sugawara_monomial},
 \begin{equation}\label{e16}
(w^{i})^{(-)}=1, \quad \deg_{-1}^{(0)}(v^1)=\deg_{-1}^{(0)}(w^i)+1,
 \end{equation}
because $(v^1)^{(-)}=1$ by \eqref{e14}.
Remember that
\begin{align}\label{e42}
(v^1)^{(+)}=e_{\be_1}(-1)^{a^{(1)}_{1,1}}\cdots e_{\be_{q}}(-1)^{a^{(1)}_{q,1}}.
\end{align}
Computing $\tilde{L}_{-1} w^{i}$, we deduce that
$$ (w^{i})^{(+)}=e_{\be_1}(-1)^{a^{(i)}_{1,1}}\cdots e_{\be_{q}}(-1)^{a^{(i)}_{q,1}}.
 $$
{Since $v^{(-)}=1$ and $\deg_{-1}^{(0)}(v^1)=\deg_{-1}^{(0)}(w^i)+1$,
it results from Lemma \ref{lem:Sugawara_bracket} (3) that $\dep((v^1)^{(+)})\geq 1$,
which contradicts (\ref{e42})}.

$\ast$ Finally, if $j \in J \setminus J_1$, then by Lemma \ref{lem:Sugawara_monomial} (b), any PBW monomial $y$ in $\tilde{L}_{-1}w^j$ satisfies that $y^{(-)}\neq 1$. So $v^1$
cannot be a PBW monomial in $\tilde{L}_{-1}w^j$.

This concludes the proof of the lemma.
\end{proof}

As already explained, Lemma \ref{lem:w_0} implies that $w$ has zero depth
and so its image in $R_{V^k(\g)}$ is nonzero,
achieving the proof of Theorem \ref{Th:main}.

\subsection{Remarks}\label{subsection:remark}
The statement of Theorem \ref{theorem:image-of-the-singular}
is not true at the critical level.
Also,
it is not true that
the depth of a depth-homogenous singular vector of $S(\g[t^{-1}]t^{-1})$ is
always zero.

Indeed,
the $\g[[t]]$-module
$S(\g[t^{-1}]t^{-1})$
can be naturally identified with $\C[J_{\infty}\g^*]$,
where $J_{\infty}X$ is the arc space of $X$,
and so 
$S(\g[t^{-1}]t^{-1})^{\g[t]}\cong \C[J_{\infty}\g^*]^{J_{\infty}G}$.
It is known \cite{RaiTau92,BeiDri,EisFre01} that
\begin{align*}
\C[J_{\infty}\g^*]^{J_{\infty}G}\cong \C[J_{\infty}(\g^*/\!/G)].
\end{align*}
This means that
the invariant ring is a polynomial ring with infinitely many
variables
$\partial^j p_i$,
$i=1,\dots, \ell$,
$j\geq 0$,
where
$p_1,\dots, p_{\ell}$ is a set of homogeneous generators of
$S(\g)^\g$
considered as elements of
$S(\g[t^{-1}]t^{-1})$ via the embedding
$S(\g)\hookrightarrow S(\g[t^{-1}]t^{-1})$,
$\g\ni x\mapsto x(-1)$.
We have  $\on{depth}(\partial^j p_i)=j$
although
each $\partial^j p_i$ is a singular vector
of $S(\g[t^{-1}]t^{-1})$.

For $k=-h^{\vee}$,
the maximal submodule $N_k$ of $V^k(\g)$
is generated by Feigin-Frenlel center (\cite{FreGai04}).
Hence \cite{FeiFre92,Fre05},
$\on{gr}N_{k}$ %\subset \on{gr}V^{-h^{\vee}}(\g)$
is exactly the
argumentation ideal of
$S(\g[t^{-1}]t^{-1})^{\g[t]}$.
Therefore,
the above argument shows that
the statement of Theorem \ref{theorem:image-of-the-singular}
is false at the critical level.

\section{$W$-algebras and proof of Theorem \ref{Th:main2}}
\label{sec:W-algebras}
Let $f$ be a nilpotent element of $\g$. By the Jacobson-Morosov theorem, 
it embeds into an $\mathfrak{sl}_2$-triple $(e,h,f)$ of $\g$.
%For $g \in G$ and $x \in\g$, write $g.x$
%for the image of $x$ by the adjoint action of $g$.
Recall that the Slodowy slice $\mathscr{S}_f$ is the affine space $f+\g^{e}$,
where $\g^{e}$ is the centralizer of $e$ in $\g$.
It has a natural Poisson structure induced from that of $\g^*$ (\cite{GanGin02}).

The embedding ${\rm span}_\C\{e,h,f\}
\cong \mathfrak{sl}_2 \hookrightarrow \g$
exponentiates to a homomorphism
$SL_2 \to G$. By restriction to the one-dimensional
torus consisting of diagonal matrices, we obtain
a one-parameter subgroup $\rho \colon \C^* \to G$.
For $t\in\C^*$ and $x\in\g$, set
\begin{align*} \label{eq:rho}
\tilde{\rho}(t)x := t^{2} \rho(t)(x).
\end{align*}
We have
$\tilde{\rho}(t)f=f$, and
the $\C^*$-action of $\tilde{\rho}$ stabilizes $\mathscr{S}_f$.
Moreover, it is contracting to $f$ on $\mathscr{S}_f$, that is, for all $x\in\g^{e}$,
$$\lim_{t\to 0} \tilde{\rho}(t)(f+x)=f.$$

The following proposition is well-known.
Since its proof is short, we give below
the argument for the convenience of the reader.

\begin{prop}[{\cite{Slo80,Premet02,Charbonnel-Moreau}}]
\label{pro:Slodowy}
The morphism
$$\theta_f \colon G \times \mathscr{S}_f \longrightarrow \g,
\quad (g,x) \longmapsto g.x$$
is smooth onto a dense open subset of $\g^*$.
\end{prop}
\begin{proof}
Since $\g =\g^{e}+[f,\g]$, the map $\theta_f$ is a submersion at
$(1_{G},f)$. Therefore, $\theta _{f}$ is a submersion at all points of
$G\times (f +\g^{e})$ because it is $G$-equivariant for the left
multiplication in $G$, and
$$ \lim _{t \to \infty }  \rho(t) .x = f$$
for all $x$ in $f+\g^{e}$. So, by~\cite[Ch.\,III, Proposition 10.4]{Hartshorne},
the map $\theta _{f}$ is a smooth morphism onto a dense open subset of
$\g$, containing $G.f$.
\end{proof}

As in the introduction, let $\W^k(\g,f)$ be the {affine $W$-algebra} associated with
a nilpotent element $f$ of $\g$
defined by the generalized quantized Drinfeld-Sokolov reduction:
$$\W^k(\g,f)=H^{0}_{DS,f}(V^k(\g)).$$
Here, $H^{\bullet}_{DS,f}(M)$ denotes the BRST
cohomology of the  generalized quantized Drinfeld-Sokolov reduction
associated with $f \in \mathcal{N}(\g)$ with coefficients in
a $V^k(\g)$-module $M$.
Recall that we have \cite{DSK06,Ara09b} a natural isomorphism
$R_{\W^k(\g,f)}\cong \C[\mathscr{S}_{f}]$ of Poisson algebras, so that
\begin{align*}
 X_{\W^k(\g,f)}= \mathscr{S}_{f}.
\end{align*}
We write $\W_k(\g,f)$ for the unique simple (graded) quotient of
$\W^k(\g,f)$. Then $X_{\W_k(\g,f)}$
 is a $\C^*$-invariant Poisson
subvariety of the Slodowy slice $\mathscr{S}_f$.

Let $\mathscr{O}_k$ be the category $\mathscr{O}$ of
$\fg$ at level $k$.
We have a functor
\begin{align*}
 \mathscr{O}_k\longrightarrow\W^k(\g,f)\on{-Mod}
 ,\quad M\longmapsto
 H^0_{DS,f}(M),
\end{align*}
where
$\W^k(\g,f)\on{-Mod}$ denotes the category
of $\W^k(\g,f)$-modules.

The full subcategory of $\mathscr{O}_k$ consisting of
objects $M$ on which $\g$ acts  locally finitely will be denoted by $\on{KL}_k$.
Note that both $V^k(\g)$ and $L_k(\g)$ are objects of $\on{KL}_k$.
 \begin{theorem}[{\cite{Ara09b}}]
 \label{Th:W-algebra-variety}
 \begin{enumerate}
  \item $H_{DS,f}^i(M)=0$ for all $i\ne 0$, $M\in
	\on{KL}_k$.
	In particular, the functor
	$\on{KL}_k\longrightarrow\W^k(\g,f)\on{-Mod}$, $M\mapsto
	H_{DS,f}^{0}(M)$, is exact.
\item
     For any quotient $V$ of $V^k(\g)$,
    \begin{align*}
X_{H^{0}_{DS,f}(V)}=X_{V}\cap  \mathscr{S}_{f}.
%\label{:eq:var-of-reduction}
\end{align*}
In particular
%\begin{enumerate}
%\item
$H_{DS,f}^{0}(V)
	  \ne 0$ if and only if
	  $\overline{G.f}\subset X_V$.
%\item $H_{DS,f}^{0}(V)$ is lisse if $X_V=\overline{G.f}$.
%\end{enumerate}
  \label{item:intersection}
    \end{enumerate}
\end{theorem}

By Theorem \ref{Th:W-algebra-variety} (1),
$H^{0}_{DS,f}(L_k(\g))$ is a quotient vertex algebra of $\W^k(\g,f)$ if it is nonzero.
Conjecturally \cite{KacRoaWak03,KacWak08},
we have
\begin{align*}
\W_k(\g,f)\cong H^{0}_{DS,f}(L_k(\g))\text{ provided that }H^{0}_{DS,f}(L_k(\g))\ne 0.
\end{align*}
(This conjecture has been verified in many cases \cite{Ara05,Ara07,Ara08-a,AEkeren19}.)

\begin{proof}[Proof of Theorem \ref{Th:main2}]
%For the critical level $k = - h^\vee$, it is known
%that $$X_{H^{0}_{DS,f}(L_k(\g))}=\mathcal{N}\cap \mathscr{S}_f \not= \mathscr{S}_f.$$
%Hence, there is no loss of generalities in assuming that $k + h^\vee \not=0$.
The directions (1) $\Rightarrow$ (2)
and  (2) $\Rightarrow$ (3) are obvious.
Let us show that (3) implies (1).
So suppose that  $X_{H^{0}_{DS,f}(L_k(\g))}=\mathscr{S}_f$.
%We wish to show
%$H^{0}_{DS,f}(V^k(\g))=H^{0}_{DS,f}(L_k(\g))$.
By Theorem~\ref{MainTheorem}, it is enough to show that $X_{L_k(\g)}=\g^*$.
Assume the contrary. Then $X_{L_k(\g)}$ is contained in a proper
$G$-invariant closed subset of $\g$.
On the other hand, by Theorem \ref{Th:W-algebra-variety} and our hypothesis,
we have
$$ \mathscr{S}_f=X_{H^{0}_{DS,f}(L_k(\g))}= X_{L_k(\g)} \cap \mathscr{S}_f.$$
Hence, $ \mathscr{S}_f$ must be  contained
in a proper
$G$-invariant closed subset of $\g$. But this contradicts Proposition \ref{pro:Slodowy}.
The proof of the theorem is completed.
\end{proof}

\bibliographystyle{alpha}

%\bibliography{/Users/tomoyuki/Documents/Dropbox/bib/math}
\end{document}